\documentclass[12pt]{article}
\usepackage{amssymb,amsthm,amsmath}
\usepackage{enumerate}

\newcommand{\dd}{\mathrm{d}}

\newcommand{\E}{\mathbb{E}}
\newcommand{\1}{\textbf{1}}
\newcommand{\R}{\mathbb{R}}

\newtheorem{theorem}{Theorem}
\newtheorem{lemma}[theorem]{Lemma}

\theoremstyle{remark}
\newtheorem{remark}[theorem]{Remark}

\theoremstyle{definition}

\def\s{\sigma}

\newcommand{\set}[1]{\left\{#1\right\}}

\renewcommand{\Pr}[1]{{\bf P}(#1)}
\begin{document}
\title{A note on log-concave random graphs}
\author{Alan Frieze\thanks{Research supported in part by NSF Grant DMS1661063} and Tomasz Tkocz\\
Department of Mathematical Sciences,\\
Carnegie Mellon University,\\
Pittsburgh PA15213,\\
U.S.A.}
\maketitle
\begin{abstract}
We establish a threshold for the connectivity of certain random graphs whose (dependent) edges are determined by the uniform distributions on generalized Orlicz balls, crucially using their negative correlation properties. We also show the existence of a unique giant component for such random graphs.
\end{abstract}
\section{Introduction}
Probabilistic combinatorics is today a thriving field bridging the classical area of probability with modern developments in combinatorics. The theory of random graphs, pioneered by Erd\H{o}s-R\'enyi \cite{ErdosRenyi59}, \cite{ErdosRenyi60} has given us numerous insights, surprises and techniques and has been used to count, to establish structural properties and to analyze algorithms. There are by now several texts \cite{B}, \cite{JLR}, \cite{FK} that deal exclusively with the subject. The most heavily studied models being $G_{n,m}$ and $G_{n,p}$. Both have vertex set $[n]$ and in the first we choose $m$ random edges and in the second we include each possible edge independently with probability $p$. 

Let $X$ be a random vector in $[0,\infty)^{\binom{n}{2}}$ with a log-concave down-monotone density $f$, that is (i) $\log f$ is concave and (ii) $f(x) \geq f(y)$ if $x \leq y$ (coordinate-wise). For $0 < p < 1$, let $G_{X,p}$ be a random graph with vertices $1,\ldots, n$ and edges determined by $X$: for $1 \leq i < j \leq n$, $\{i,j\}$ is an edge if and only if $X_{\{i,j\}} \leq p$. Such \emph{log-concave random graphs} were introduced by Frieze, Vempala and Vera in \cite{FVV}. For instance, when $X$ is uniform on $[0,1]^{\binom{n}{2}}$, $G_{X,p}$ is the  random graph $G_{n,p}$. 
The paper \cite{FVV} introduced a surprising connection between random graphs and convex geometry.

It studied, among other things, the connectivity of $G_{X,p}$ and found a logarithmic gap for the threshold. There is no gap when $G_{X,p}$ is defined by uniform sampling from a ``well-behaved'' regular simplex\footnote{A regular simplex $\set{x\in \R^d:a\cdot x\leq 1}$ for some $a\geq 0$ if $a_i/a_j\leq K$ for some not too large $K$.} and we extend this case to {\em Generalized Orlicz Balls} GOBs: that is sets of the form $\{x \in \R^d: \sum_{i=1}^d f_i(|x_i|) \leq 1\}$ for some nondecreasing lower semicontinuous convex functions $f_1,\ldots,f_d: [0,\infty)\to [0,\infty]$ with $f_i(0) = 0$, which are not identically $0$ or $+\infty$ on $(0,\infty)$. 

The key property of Orlicz balls is {\em negative correlation}. We say that a random vector $X$ in $\R^d$ has negatively correlated coordinates if for any disjoint subsets $I$, $J$ of $\{1,\ldots,d\}$ and nonnegative numbers $s_i$, $t_j$, we have 
\[
\Pr{\forall i \in I \ |X_i| > s_i, \forall j \in J \ |X_j| > t_j} \leq \Pr{\forall i \in I \ |X_i| > s_i}\Pr{\forall j \in J \ |X_j| > t_j}.
\]
It was shown in \cite{PW2} that this property holds for random vectors uniformly distributed on GOBs (see also \cite{PW1} for a first such result treating two coordinates and \cite{Woj} for a simpler proof of the general result).

{\bf Notation:} Throughout the paper we will let $\s_{\min}$ and $\s_{\max}$ be defined by  
$$\sigma_{\min}^2 = \min_{1 \leq i < j \leq n}\E X_{i,j}^2 \qquad \text{and} \qquad \sigma_{\max}^2 = \max_{1 \leq i < j \leq n}\E X_{i,j}^2.$$
Our result concerning connectivity is the following theorem.
\begin{theorem}\label{th1}
Let $X = (X_{i,j})_{1\leq i < j \leq n}$ be a log-concave random vector in $[0,\infty)^{\binom{n}{2}}$ with a down-monotone density and negatively correlated coordinates. 
\begin{enumerate}[(a)]
\item For every $\delta \in (0,1)$, there are constants $c_1$ and $c_2$ dependent only on $\delta$ such that for $p < c_1\sigma_{\min}\frac{\log n}{n}$, we have
\[
\Pr{G_{X,p} \textrm{ has isolated vertices}} > 1 - c_2n^{-\delta}.
\]
\item For every $\delta \in (0,1)$, there are constants $C_1$ and $C_2$ dependent only on $\delta$ such that for $p > C_1\sigma_{\max}\frac{\log n}{n}$, we have
\[
\Pr{G_{X,p} \textrm{ is connected }} > 1 - C_2n^{-\delta}.
\]
\end{enumerate}
\end{theorem}
We will also discuss the existence of a giant component for smaller values of $p$. 

{\bf Notation:} Let
\begin{equation}\label{eq:def-M}
M = \max_{T}\sup_{y \in [0,\infty)^{|T|}}\max_{(i,j) \notin T} \E(X_{i,j}^2|X_T = y),
\end{equation}
where the first maximum is over all nonempty subsets $T$ of the index set $\{(i,j), \ 1 \leq i < j \leq n\}$ and we denote $X_T = (X_{i,j})_{(i,j) \in T}$. 

For our theorem on the existence of a giant component we need to have $M=O(1)$. For a GOB, $\set{x \in \R^{\binom{n}{2}}: \sum_{1\leq i<j\leq n} f_{i,j}(|x_{i,j}|) \leq 1}$, this is justified by the following assumption: we let
$$a_{i,j}=\sup\set{t>0:f_{i,j}(t)\leq 1}.$$
Now our assumptions on the $f_{i,j}$ imply that the $a_{i,j}$'s are finite. Furthermore, $M\leq \max_{i,j}a_{i,j}^2$ and so our assumption here is that $\max_{i,j}a_{i,j}$ is bounded by an absolute constant.
\begin{theorem}\label{th2}
Let $X = (X_{i,j})_{1\leq i < j \leq n}$ be a log-concave random vector in $[0,\infty)^{\binom{n}{2}}$ with a down-monotone density. Assume that 
$M=O(1)$.  There are constants $c_1$ and $c_2$ such that for every $\beta > 1$, we have
\begin{enumerate}[(i)]
\item\label{lm:giant-smallp}
If $p < \frac{c_1\sigma_{\min}}{n}$, then 
\[
\Pr{G_{X,p} \textrm{ has a component of order } \geq \beta\log n } < \frac{12}{n^{\beta-1}}.\]
\item\label{lm:giant-largep}
If $p > \frac{c_2M\log\left(\frac{M}{\sigma_{\min}}\right)}{n}$, then
\[
\Pr{G_{X,p} \textrm{ has a component of order } \in [\beta\log n,n/2] } < \frac{1}{n^{\beta-1}}
\]
and
\[
\Pr{G_{X,p} \textrm{ has a unique giant component of order } > n/2} > 1 - \frac{5\beta}{\log n}-\frac{1}{n^{\beta-1}}.
\]
\end{enumerate}
\end{theorem}
Note that we have dropped the assumption of negative correlation.
\section{Connectivity: Proof of Theorem \ref{th1}}
\begin{proof}
Part (b) is part of Theorem 2.1 of \cite{FVV}. For (a), we adapt the standard second moment argument used for the Erd\"os-R\'enyi model. For $1 \leq i \leq n$, let $Y_i$ be equal to $1$ if the vertex $i$ is isolated and $0$ otherwise. Let $Y = Y_1 + \ldots + Y_n$ be the number of isolated vertices. We have,
\[
\Pr{G_{X,p} \textrm{ has isolated vertices}}=  \Pr{Y  > 0} \geq \frac{(\E Y)^2}{\E Y^2}.
\]
Thus, if we show that $\E Y^2 \leq (1+\varepsilon)(\E Y)^2$, then $\Pr{Y > 0} \geq 1 - \varepsilon$. Clearly,
\[
\E Y^2 = \sum_{k} \E Y_k^2 + \sum_{k \neq l} \E Y_k Y_l = \sum_k \E Y_k + \sum_{k \neq l} \Pr{Y_k = 1 = Y_l} = \E Y + \sum_{k \neq l} \Pr{Y_k = 1 = Y_l}
\]
and our goal is to show that
\[
\E Y \leq \frac{\varepsilon}{2}(\E Y)^2 \qquad \text{and} \qquad \sum_{k \neq l} \Pr{Y_k = 1 = Y_l} \leq \left(1+\frac{\varepsilon}{2}\right)(\E Y)^2.
\]
From the negative correlation of coordinates of $X$ as well as an elementary inequality $\Pr{A} \leq \Pr{A\cap B} + 1 - \Pr{B}$, we get
\begin{align*}
\Pr{Y_k=1=Y_l} &= \Pr{\forall i \neq k \ X_{ik} > p, X_{il} > p, X_{kl} > p} \\
&\leq \Pr{\forall i \neq k \ X_{ik} > p}\Pr{\forall i\neq k,l \ X_{il} > p} \\
&\leq \Pr{\forall i \neq k \ X_{ik} > p}\big[ \Pr{\forall i\neq l \ X_{il} > p} + 1 - \Pr{X_{kl} > p} \big] \\
&= \Pr{Y_k = 1}\big[ \Pr{Y_l = 1} + \Pr{X_{kl} \leq p} \big].
\end{align*}  
By Lemma 3.5 from \cite{FVV}, $\Pr{X_{kl} \leq p} \leq \frac{p}{\sigma_{\min}}$ (recall that by the Pr\'ekopa-Leindler inequality, marginals of log-concave vectors are log-concave; clearly, marginals of down-monotone densities are down-monotone). Therefore,
\begin{align*}
\sum_{k \neq l} \Pr{Y_k = 1 = Y_l} &\leq \sum_{k \neq l} \Pr{Y_k = 1}\Pr{Y_l =1} + \sum_{k\neq l} \Pr{Y_k = 1}\frac{p}{\sigma_{\min}} \\
&\leq \left(\sum_{k} \Pr{Y_k=1}\right)^2 + \frac{np}{\sigma_{\min}}\sum_{k} \Pr{Y_k = 1} \\
&\leq \left(1 + \frac{np}{\sigma_{\min}\E Y}\right)(\E Y)^2  < \left(1 + \frac{c_1\log n}{\E Y}\right)(\E Y)^2,
\end{align*}
so it suffices to take $\varepsilon$ such that
\[
\varepsilon \geq \frac{2}{\E Y} \qquad \text{and} \qquad \varepsilon \geq \frac{2c_1\log n}{\E Y}.
\]
By Lemma 3.1 from \cite{FVV}, $\Pr{Y_k = 1} \geq e^{-a p n /\sigma_{\min} }$, for some universal constant $a$ (the assumption $p < \frac{1}{4}\sigma_{\min}$ of that lemma is clearly satisfied if $p < c_1\sigma_{\min}\frac{\log n}{n}$), so
\[
\E Y = \sum_k \Pr{Y_k = 1} \geq ne^{-a p n /\sigma_{\min} } > n^{1-ac_1}.
\]
Thus, $\varepsilon = c_2n^{ac_1-1}\log n$ will suffice.
\end{proof}

\section{Giant Component: Proof of Theorem \ref{th2}}
\begin{lemma}\label{lm:inout-edge-uppbd}
Let $X = (X_{i,j})_{1\leq i < j \leq n}$ be a log-concave random vector in $[0,\infty)^{\binom{n}{2}}$ with a down-monotone density. There are universal constants $a$ and $b$ such that for $S, T \subset \{(i,j), \ 1 \leq i < j \leq n\}$ and $p > 0$, we have
\[
\Pr{\forall s \in S \ X_s > p, \forall t \in T \ X_t \leq p} \leq e^{-ap|S|/M}\left(\frac{bp}{\sigma_{\min}}\right)^{|T|}.
\]
\end{lemma}
\begin{proof}
Fix disjoint sets $S, T \subset \{(i,j), \ 1 \leq i < j \leq n\}$ (if they are not disjoint, the probability in question is $0$) and $y \in [0,\infty)^{|T|}$. Let $f$ be the density of $(X_S,X_T)$. The conditional density of the vector $X_S$ given $X_T = y$,
\[
f_{X_S|X_T}(x|y) = \frac{f(x,y)}{\int f(x',y)\dd x'}
\]
is down-monotone and log-concave. Therefore, by Lemma 3.1 from \cite{FVV},
\[
\Pr{\forall s \in S \ X_s > p| X_T=y} \leq e^{-ap|S|/M}.
\]
We denote the density of $X_T$ by $f_{X_T}$ and get
\begin{align*}
\Pr{\forall s \in S \ X_s > p, \forall t \in T \ X_t \leq p} &= \int_{[0,p]^{|T|}} \Pr{\forall s \in S \ X_s > p | X_T = y} f_{X_T}(y) \dd y \\
&\leq \int_{[0,p]^{|T|}}e^{-ap|S|/M} f_{X_T}(y) \dd y \\
&= e^{-ap|S|/M}\Pr{\forall t \in T \ X_t \leq p}\\
&\leq e^{-ap|S|/M}\left(\frac{bp}{\sigma_{\min}}\right)^{|T|},
\end{align*}
where the final inequality follows directly from Lemma 3.2 of \cite{FVV}.
\end{proof}

With this lemma in hand, we can prove Theorem \ref{th2}.
\begin{proof}
Let $Z_k$ be the number of components of order $k$ (that is, on $k$ vertices) in $G_{X,p}$. As for the Erd\"os-R\'enyi model, looking at a spanning tree for each component and bounding the corresponding in-out edge probabilities using Lemma \ref{lm:inout-edge-uppbd} yields
\begin{align*}
\E Z_k &\leq \binom{n}{k}k^{k-2}e^{-apk(n-k)/M}\left(\frac{bp}{\sigma_{\min}}\right)^{k-1} \\
&\leq \left(\frac{en}{k}\right)^kk^{k-2}e^{-apk(n-k)/M}\left(\frac{bp}{\sigma_{\min}}\right)^{k-1} \\
&=\frac{\sigma_{\min}}{bp}\frac{1}{k^2}\left[\frac{eb}{\sigma_{\min}}pne^{-\frac{ap}{M}(n-k)}\right]^k.
\end{align*}
If $p = \frac{M}{a}\frac{c}{n}$, with $c$ being a constant (chosen soon), this becomes
\[
\E Z_k \leq \frac{e}{A}\frac{1}{c}\frac{n}{k^2}\left[Ace^{-c}e^{ck/n}\right]^k,
\]
where we put $A = \frac{eb}{a}\frac{M}{\sigma_{\min}}$.

\bigskip
\emph{Case 1.}
If $c$ is a small constant, say $c \leq \frac{1}{eA}$ (equivalently, $p \leq \frac{\sigma_{\min}}{e^2b}\frac{1}{n}$), then we bound $e^{-c}e^{ck/n}$ crudely by $1$ and get that
\[
\E Z_k \leq \frac{e}{A}\frac{1}{c}\frac{n}{k^2}(Ac)^{k} \leq en(Ac)^{k-1} \leq e^2ne^{-k}.
\]
Thus,
\[
\E \left(\sum_{k \geq \beta \log n} Z_k \right) \leq e^2n\cdot \sum_{k \geq \beta \log n}e^{-k} \leq e^2ne^{-\beta\log n}\frac{1}{1 - e^{-1}} = \frac{e^3}{e-1}\frac{1}{n^{\beta-1}} < \frac{12}{n^{\beta-1}}.
\]
By the first moment method, this gives \eqref{lm:giant-smallp}.

\bigskip\noindent
\emph{Case 2.}
Let $c$ be a large constant, say such that $Ace^{-c/2} \leq \frac{1}{e}$ and $Ac \geq e^2$, which holds when, say $c \geq 4\log A$, provided that $A$ is large enough, which leads to the assumption on $p$ in \eqref{lm:giant-largep}. Then for $k \leq n/2$, we have
\[
\E Z_k \leq \frac{en}{Ac} (Ace^{-c/2})^k \leq ne^{-k-1}.
\]
Thus,
\[
\E \left(\sum_{\beta \log n \leq k \leq n/2 } Z_k \right) \leq ne^{-1}\sum_{k \geq \beta \log n} e^{-k} \leq \frac{1}{e-1}\frac{1}{n^{\beta-1}} < \frac{1}{n^{\beta-1}}.
\]
By the first moment method, this gives the first part of \eqref{lm:giant-largep}.

To go about the second part and show that there is a giant component, we shall simply count the number of vertices on the small components and show that with high probability, there are strictly less $n$ such vertices. The uniqueness of a giant component plainly follows from the fact that it has more than $n/2$ vertices, so there cannot be more than one such components. Fix $1 \leq k \leq \beta\log n$ and set $t = ne^{-k-1}$. For any positive integer $l \leq et+1$, we have
\begin{align*}
\Pr{Z_k \geq et} &\leq \Pr{Z_k(Z_k-1)\ldots(Z_k-l+1) \geq et(et-1)\ldots(et-l+1) } \\
&\leq \frac{\E Z_k(Z_k-1)\ldots(Z_k-l+1)}{et(et-1)\ldots(et-l+1)} \\
&\leq
\frac{\E Z_k(Z_k-1)\ldots(Z_k-l+1)}{(et-l+1)^l}.
\end{align*}
As for the upper bound for $\E Z_k$, looking at spanning trees for each $l$-tuple of distinct components of order $k$ and bounding the corresponding in-out edge probabilities using Lemma \ref{lm:inout-edge-uppbd} yields
\begin{align*}
\E Z_k&(Z_k-1)\ldots(Z_k-l+1) \\
&\leq \binom{n}{k}\binom{n-k}{k}\ldots\binom{n-(m-1)k}{k}(k^{k-2})^le^{-\frac{ap}{M}kl(n-kl)}\left(\frac{bp}{\sigma_{\min}}\right)^{(k-1)l} \\
&\leq \left(\frac{en}{k}\right)^{kl}(k^{k-2})^le^{-\frac{ap}{M}kl(n-kl)}\left(\frac{bp}{\sigma_{\min}}\right)^{(k-1)l} \\
&= \left(\frac{e}{A}\frac{1}{c}\frac{n}{k^2}\left[Ace^{-c}e^{ckl/n}\right]^k\right)^l.
\end{align*}
Provided that $kl \leq n/2$, under our assumption $c\geq 4\log A$, this is further upper bounded by $(t/k^2)^l$, which gives
\[
\Pr{Z_k \geq ne^{-k}} = \Pr{Z_k \geq et} \leq \frac{1}{k^{2l}}\left(\frac{t}{et-l+1}\right)^l.
\]
For $k \geq \frac{1}{2}\log n$, we choose $l = 1$ and get
\[
\Pr{Z_k \geq ne^{-k}} \leq \frac{1}{e}\frac{1}{(\frac{1}{2}\log n)^2}, \qquad k \geq \frac{1}{2}\log n.
\]
For $k < \frac{1}{2}\log n$, we have $t = ne^{-k-1} > e^{-1}\sqrt{n}$, so choosing, say $l - 1 = \lfloor e^{-1}\sqrt{n} \rfloor$ yields
\begin{align*}
\Pr{Z_k \geq ne^{-k}} \leq \left(\frac{t}{et-\lfloor e^{-1}\sqrt{n} \rfloor}\right)^l = \left(\frac{1}{e-\frac{\lfloor e^{-1}\sqrt{n} \rfloor}{t}}\right)^l &\leq \left(\frac{1}{e-1}\right)^{l} \\
&\leq \left(\frac{1}{e-1}\right)^{e^{-1}\sqrt{n}}, \qquad k < \frac{1}{2}\log n.
\end{align*}
Combining the last two estimates, the union bound gives that the probability of the event $E = \{\exists k \leq \beta\log n, \ Z_k \geq ne^{-k}\}$ is at most
\[
\frac{4}{e}\frac{(\beta-1/2)\log n+1}{(\log n)^2} + \frac{\frac{1}{2}\log n}{(e-1)^{e^{-1}\sqrt{n}}} < \frac{5\beta}{\log n}
\]
(we check that $ \frac{\frac{1}{2}\log n}{(e-1)^{e^{-1}\sqrt{n}}} < \frac{2}{\log n}$ and simply bound $\frac{4}{e}\frac{(\beta-1/2)\log n+1}{(\log n)^2} \leq \frac{\frac{4}{e}\beta+\frac{2}{e}}{\log n}$). To finish, it remains to check that on $E^c$, there are few vertices on the small components. On $E^c$, we have
\[
\sum_{k \leq \beta \log n} kZ_k \leq n\sum_{k \leq \beta \log n} ke^{-k} < n \sum_{k=1}^\infty ke^{-k} =n\frac{e}{(e-1)^2} < 0.93n.
\]
\end{proof}

\begin{remark}\label{rem:moregen}
It was shown in \cite{Woj} that the negative correlation property holds in fact for random vectors with densities of the form $h(\sum f_i(x_i))$, where $h:[0,\infty)\to[0,\infty)$ is a nonincreasing log-concave function ($h = \1_{[0,1]}$ giving uniform densities on GOBs). For such densities, $M$ is finite and can be bounded as for GOBs in terms of certain parameters depending on the functions $f_i$ and $h$.
\end{remark}

\section{Conclusion and Open Questions}
We have successfully generalised the results on the regular simplex in \cite{FVV} to GOBs. The following questions seem most apposite.
\begin{enumerate}[{\bf Q1}]
\item What we prove in Theorem \ref{th2} does not rule out the possibility that in some range of $p$ there is more than one giant component. Can the proof be tightened to rule this out?
\item What is the connectivity or giant component threshold for the intersection of two well-behaved regular simplices?
\item What is the connectivity or giant component threshold for the intersection of a {\em few} regular simplices with independent randomly chosen coefficients?
\end{enumerate}

\end{document}